\documentclass[11pt]{amsart}
\usepackage{amssymb,upref}
\usepackage[mathcal]{euscript}
\usepackage[cmtip,all]{xy}

\newtheorem{theorem}[subsection]{Theorem}

\newtheorem{definition}[subsection]{Definition}
\newtheorem{example}[subsection]{Example}
\newtheorem{lemma}[subsection]{Lemma}
\newtheorem{proposition}[subsection]{Proposition}

\title{Representations of Generalized Almost-Jordan algebras }

\author [Marcelo Flores]{Marcelo Flores}
 \address{ Departamento de Matem\'aticas,
 Facultad  de Ciencias, Universidad de Valpara\'{\i}so.
Err\'azuriz 1834, Valpara\'{\i}so, Chile.}
\email{mfloreshenriquez@gmail.com}

\author[Alicia Labra]{Alicia Labra$^{\star}$}
\thanks{$^{\star}$Supported by Fondecyt 1120844.}
\address{Departamento de Matem\'aticas,
Facultad de Ciencias, Universidad de Chile.  Casilla 653, Santiago, Chile}
\email{alimat@uchile.cl}

\date{}

\begin{document}

%\maketitle

\begin{abstract}
This paper deals with the variety of  commutative
algebras satisfying the identity $ \beta \{(yx^2)x - ((yx)x)x\} + \gamma \{yx^3 - ((yx)x)x \}=0$ where $\alpha, \beta$ are scalars.
These algebras appeared as one of the four families of degree four identities in Carini, Hentzel and Piacentini-Cattaneo \cite{CHP}.
We give a  characterization of representations and irreducible modules on these algebras. Our results require that the characteristic
of the ground field was different from $2,3.$
\end{abstract}

\medskip

\maketitle

%------------------------------------------

\section{Introduction}

Let $A$ be a commutative not necessarily associative algebra over an infinite  field $F.$
Let $x$ be an element in $A.$ We define the principal powers of $x$ by $x^1 = x, \; x^{n+1} = x^n x$ for all $n \geq 1.$

A {\it Jordan algebra} is a commutative algebra satisfying the identity $x^2(yx) -(x^2y)x = 0.$  It is a well known variety of algebra, that is {\it power-associative,} i.e., the subalgebra generated by any element of the algebra, is associative. See \cite{ Jacobson 1}, \cite{Schafer 2} for properties of these varieties of algebras. It is known (see  \cite{Osborn 1}) that a Jordan algebra satisfies the identity $3((yx^2)x) = 2 ((yx)x)x + yx^3.$ These algebras, called {\it almost-Jordan algebras}  have been studied by Osborn \cite{Osborn 1}, \cite{Osborn 2}, Petersson \cite{Petersson}, Sidorov  \cite{Sidorov}, and Hentzel and Peresi  \cite{HP}. In this last paper, the authors proved that every semi-prime almost-Jordan algebra is a Jordan algebra and this fact justified the name of these algebras.

\medskip
 A {\it generalized almost-Jordan algebra} is a commutative algebra satisfying the
identity
\begin{equation}\label{1}
    \beta \{(yx^2)x - ((yx)x)x\} + \gamma \{yx^3 - ((yx)x)x \}=0
\end{equation}
for every x,y $\in A$ where $\alpha, \beta$ are scalars, and $(\beta, \gamma) \neq (0,0).$
For $\beta = 3$ and $\gamma = -1,$ we have an almost-Jordan algebra.

\medskip
In the study of degree four identities not implied by conmutativity, Osborn \cite{Osborn 2} classified those that were implied by
 the fact of possessing a unit element.  Carini, Hentzel and Piacentini-Cattaneo \cite{CHP} extended this work by dropping the
 restriction on the existence of the unit element.
This result require that  characteristic $F \neq 2,3.$  The identity defining a generalized almost-Jordan algebra with $\beta, \gamma \in F$ appears
as one of these identities.

 We observe that there are generalized almost-Jordan algebras
that are not Jordan algebras.

\begin{example}\label{ej1}
Let $A$ be a commutative algebra over $F$ with base $\{e,a\}$, and multiplication table given by
$$e^2 = e, \; a^2 =  e, \;  \: \mbox{all other products being zero}$$
Then $A$ satisfies identity (\ref{1}) for  $\beta=1$ y $\gamma=-1$. Moreover  $A$ is not a Jordan algebra, since it is not power-associative since $a^2(aa)\not= (a^2a)a$.
\end{example}

\begin{example} \label{ej2}
Let  $A$ be a commutative algebra over $F$ with base $\{e,a\}$ and multiplication table given by
$$e^2 = e, \; ea = ae =  -e - a, \;  a^2 =  e+ a. \: $$
Then $A$ satisfies identity (\ref{1}) with $\beta=0$ and $\gamma\not=0$. That  is $A$ satisfies  $x^3y=((xy)x)x$ for every
$x,y \in A$. Moreover  $A$ is not a Jordan algebra, since $a^2(aa) = 2a\not= (a^2a)a = 0$.
\end{example}

\begin{example}\label{ej3}
Let $A$ be a commutative $F$-algebra with base $\{e,a,b\}$ and multiplication table given by
$$e^2 = e, \; ab = ba = b, \;  \: \mbox{all other products being zero} $$
Then $A$ satisfies identity (\ref{1}) with $\beta=1$ and $\gamma=1$, for every  $\alpha \in F$. Moreover $A$ is not power-associative since $(a+b)^4 = 2b$ and $(a+b)^2 (a+b)^2 = 0$ and $A$ is not a  Jordan
algebra
\end{example}

\bigskip
Generalized almost-Jordan algebras $A$ have been studied in \cite{AL}, where the authors  proved that these algebras always have a trace form in terms of the trace of right multiplication operators. They also prove that if $A$ is finite-dimensional and solvable, then it is nilpotent and found three conditions, any of which implies that a finite-dimensional right-nilalgebra $A$ is nilpotent.  In \cite{Arenas}  the author found the Wedderburn decomposition of $A$ assuming that  for every ideal $I$ of $A$ either $ I$ has a non zero idempotent or $I \subset R, R$ the solvable radical of $A$ and the quotient $A /R$ is separable and in \cite{HL} where, assuming that $A$ also satisfies $((xx)x)x = 0$  the authors proved the existence of an ideal $I$ of $A$ such that $AI = IA= 0$ and the quotient algebra $A/I$ is power-associative.

In this paper we deal with representations of algebras.  Let $A$ be an algebra which belongs to a class ${\mathcal C}$
of commutative algebras over a field $K$ and let $M$ be a vector
space over $F.$ As in Eilenberg \cite{Eilenberg}, we say that a linear map
$\rho: A \rightarrow End(M)$
 is a  {\it a representation of $A$ in the class $\mathcal C$} if the split null
extension $ S = A \oplus M $ of $M,$ with
multiplication given by
$$(a + m)(b + n) = ab + \rho(a)(n) +
\rho(b)(m) \; \; \forall \; a,b \in A, m,n \in M $$
 belongs to the class $\mathcal C$.

Representations have been studied for different algebras, for example,  in \cite{Jacobson 3}, and \cite{Osborn 4} for Jordan algebras, in\cite{Jacobson 2}, \cite{Schafer 1} and \cite{Shestakov} for alternative algebras, in  \cite{McCrimmon} for composition algebras, in \cite{Humphreys} for Lie algebras, in \cite{Sidorov} for Lie triple algebras, in \cite{Osborn 3} for Novikov algebras, in \cite{BIM} for Bernstein algebras, in \cite{LR} for train algebras of rank 3, in \cite{BLR} for power-associative train algebras of rank 4, in \cite{BL} for algebras of rank 3, in \cite{ES 1} for Malcev algebras and in \cite{ES 2} for Malcev super algebras.

A representation $\rho: A \rightarrow End(M)$ is said to be {\it
irreducible} if $M \neq 0$ and there is no proper subspace of $M$
which is invariant under all the transformations $\rho(a)$, $a \in A$,
and is said to be {\it $r$-dimensional} if $\dim M = r$.

In this paper we study representations and irreducible modules over generalized almost-Jordan algebras $A$.  The paper is organized as follows.
 In Section $2$ we find necessary and sufficient conditions
for a linear map $\rho$ to be a representation. We find the action of $A$  over $M$ when $0\not \in \{\gamma,\beta+ \gamma,\beta+2\gamma \}$ (see Theorem \ref{teorema 1}).
In Section $3$ we look at the three cases that arose as exception in the above Theorem. In Section $4$ we study irreducible modules over generalized almost-Jordan algebras
and we prove two theorems when $0\not \in \{\beta, \gamma,\beta+ \gamma,\beta+2\gamma,  \beta+2\gamma\}$ (see Theorem \ref{teoasc} and 
Theorem \ref{teoasc2}).  The last Theorem need also the condition $ \beta- \gamma \neq 0.$ In Section $5$ we look at five cases that arose as exception in Theorem \ref{teoasc} and
Theorem \ref{teoasc2}. Moreover, in the case  $\beta + \gamma= 0$ we give an example of a $2$-dimensional irreducible module $M.$ Finally we present some open problems. 

%------------------------------------------

\section{Representations}

\medskip
In this section we study  representations of generalized almost-Jordan algebras.

\begin{lemma}\label{lema 1}
Let $A$ be a generalized almost-Jordan algebra and
$ \rho: A \rightarrow End(M) $ a  linear map. Then $\rho$ is a representation of $A$, if and only if for every $a,b \in A$ the following identities hold
\begin{equation}\label{repre1}
    (\beta+\gamma)\rho_{a}^{3}-\beta \rho_a\rho_{a^2}-\gamma\rho_{a^3}=0
  \end{equation}
  \begin{equation}\label{repre2}
    (\beta +\gamma)(\rho_a\rho_{ab}+\rho_a^2\rho_b+\rho_{(ab)a})-\beta(2\rho_a\rho_b\rho_a +\rho_{a^2b})-\gamma(2\rho_b\rho_a^2+\rho_b\rho_{a^2})=0
  \end{equation}
where $\rho_a:=\rho(a)\in End(M)$, and for every $a,b \in A$, $\rho_a\circ\rho_b$ will be  denoted by $\rho_a\rho_b$.
\end{lemma}

%\end{document}
%(\ref{repre1})

\begin{proof}
 $\rho$ is a representation of  $A$ if and only if every $a+m, b+n \in A\oplus M$ satisfy the identity (\ref{1}).  Straightforward calculations give
 \begin{equation}\nonumber
    [(a+m)^2(b+n)](a+m)=(a^2b)a + 2\rho_a(\rho_b(\rho_a(m)))+\rho_a(\rho_{a^2}(n))+\rho_{a^2b}(m)
 \end{equation}
  \begin{equation}\nonumber
    [(a+m)^3](b+n)=a^3b + 2\rho_b(\rho_a(\rho_a(m)))+\rho_b(\rho_{a^2}(m))+\rho_{a^3}(n)
 \end{equation}

$$[[(a+m)(b+n)](a+m)](a+m)=((ab)a)a + \rho_a(\rho_a(\rho_b(m)))+\rho_a(\rho_{ab}(m)) +$$
$$\rho_{a}^3(n)+\rho_{(ab)a}(m)$$
Replacing $x=a+m, y=b+n$ in identity (\ref{1})  we get
    $$\beta \{2\rho_a(\rho_b(\rho_a(m)))+\rho_a(\rho_{a^2}(n))+\rho_{a^2b}(m)- \rho_{a}^2((\rho_b(m))-\rho_a(\rho_{ab}(m))-\rho_{a}^3(n)- $$
$$\rho_{(ab)a}(m) \}
      + \gamma \{2\rho_b(\rho_{a}^2(m))+\rho_b(\rho_{a^2}(m))+\rho_{a^3}(n) - \rho_{a}^2((\rho_b(m))-\rho_a(\rho_{ab}(m))-$$
$$\rho_{a}^3(n)-\rho_{(ab)a}(m)\}=0$$

Now it is  easy to see that this relation holds if and only if identities (\ref{repre1}) and (\ref{repre2}) hold in $A.$
\end{proof}

%\end{document}

In the following suppose that $A$ has an  idempotent element $e \neq 0$. Taking  $a=e$ in identity  (\ref{repre1}), we obtain
\begin{equation}\label{poli_car}
     (\beta+\gamma)\rho_{e}^{3}-\beta \rho_{e}^2-\gamma\rho_{e}=0
\end{equation}

\begin{proposition}\label{pr:PeirceM}
Let $A$ be a generalized almost-Jordan algebra and  $\beta, \gamma$ satisfying $0\not \in \{\gamma,\beta+ \gamma,\beta+2\gamma \}$. Suppose that  $A$ has an  idempotent element
$e\neq 0$. Let $\rho: A \rightarrow End(M)$ be  a representation of A. Then
$$M=M_0\oplus M_1 \oplus M_{\lambda}$$
where  $M_i=\{m\in M | \rho_e(m)=im \}$, and $i\in\{0,1,\lambda\}$.
\end{proposition}
\begin{proof}

Using identity (\ref{poli_car}) we see that $\rho_e$ satisfies the polynomial $p(x)=(\beta+\gamma)x^{3}-\beta x^2-\gamma x=0$. Since $\beta+\gamma \neq 0$ we define $\lambda=\frac{-\gamma}{\beta+\gamma}\in F$ and  $p(x)=(\beta+\gamma)x(x-1)(x-\lambda)$. Moreover, since $\gamma\neq 0, \beta+2\gamma \neq 0$, we have that $\lambda\neq 0$ and $\lambda\neq 1$, therefore $p(x)$ has only   simple roots.  On the other hand, since  the minimal polynomial of the operator $\rho_e$, is a divisor of $p(x)$ then  $\rho_e$, also has simple roots. Then  $\rho_e$ is diagonalizable and
$$ M=M_0\oplus M_1 \oplus M_{\lambda}$$
\end{proof}
%\end{document}

In \cite[Theorem 1]{Arenas} M. Arenas proves that the Peirce decomposition of a  generalized almost-Jordan algebra $A$ is given by  $A=A_0\oplus A_1 \oplus A_{\lambda}, \;  A_i=\{a\in A | e a=i a \}$, where $\lambda=\frac{-\gamma}{\beta+ \gamma}$. Moreover, when $0\not \in \{\gamma,\beta+ \gamma,\beta+2\gamma \}$ we have the following relations among these subspaces
  $$(A_0)^2\subseteq A_0,\   (A_1)^2\subseteq A_1,\   A_0A_1=\{0\} $$
    $$A_{\lambda}A_0\subseteq A_{\lambda},\   A_{\lambda}A_1\subseteq A_{\lambda},\    (A_{\lambda})^2\subseteq A_0\oplus A_1$$

\bigskip
 Next we obtain relations among  $A_i$ and $M_i$, where $i\in\{0,1,\lambda\}$. Linearising  identity (\ref{repre1}) we  have
 \begin{equation}\label{linrepre1}
   (\beta+\gamma)(\rho_a\rho_b\rho_a+\rho_b\rho_{a}^2+\rho_{a}^2\rho_b)-2\beta\rho_a\rho_{ab}-\beta\rho_b\rho_{a^2}-2\gamma\rho_{(ab)a}-\gamma\rho_{a^2b}=0
 \end{equation}
 Moreover sustracting identities (\ref{repre2}) and (\ref{linrepre1}), we  obtain
  %\begin{equation}\label{sum}
   % (3\beta+\gamma)\{\rho_a\rho_b\rho_a-\rho_a\rho_{ab}\} + (3\gamma+\beta)\{\rho_b\rho_{a}^{2}-\rho_{(ab)a}\} +(\gamma-\beta)\{\rho_b\rho_{a^2}-\rho_{a^2b}\}=0
 %\end{equation}
 \begin{equation}\label{dif}
    (\gamma-\beta)\{\rho_a\rho_b\rho_a+ \rho_a\rho_{ab}-\rho_b\rho_{a}^{2}-\rho_{(ab)a}\}+ (\beta+\gamma)\{2\rho_{a}^{2}\rho_b-\rho_b\rho_{a^2}-\rho_{a^2b}\}=0
 \end{equation}

 \begin{theorem}\label{teorema 1}
 Let $A$ be a generalized almost-Jordan algebra and  $\beta, \gamma$ satisfying  $0\not \in \{\gamma,\beta+ \gamma,\beta+2\gamma \}$. Suposse that  $A$ has an idempotent element  $e\neq 0$. Let $\rho: A \rightarrow End(M)$ be a representation of $A. $ Then the action of $A$ on $M$ satisfies the following relations
 $$A_0\cdot M_0 \subseteq M_0,\ A_0\cdot M_1 = \{0\},\ A_0\cdot M_{\lambda} \subseteq M_{\lambda}, $$
 $$A_1\cdot M_0 =\{0\},\ A_1\cdot M_1 \subseteq M_1,\ A_1\cdot M_{\lambda} \subseteq M_{\lambda},$$
 $$A_{\lambda}\cdot M_0 \subseteq M_{\lambda},\ A_{\lambda}\cdot M_1 \subseteq M_{\lambda},\ A_{\lambda}\cdot M_{\lambda} = M_0\oplus M_1.$$
 Moreover, if we assume that $\beta\neq0$ and $\beta+3\gamma\neq 0$, then $A_0\cdot M_{\lambda}=A_{\lambda}\cdot M_0=\{0\}$ and $ A_{\lambda}\cdot M_{\lambda}=\{0\}, $ where $\lambda=\frac{-\gamma}{\beta+ \gamma}$.
 \end{theorem}

  \begin{proof}
Since $\rho $ is  a representation of $A,$ we have that $A$ and $S=A\oplus M$, are generalized almost-Jordan algebra for the same scalars $(\beta,\gamma)\in F\times F.$
If  $e$ is an idempotent element in $S$,  since  $0\not\in \{\gamma,\beta+ \gamma,\beta+2\gamma \} $ then the Peirce decomposition of $S$ relative to   $e$ is
$S=S_0\oplus S_1\oplus S_{\lambda}$, where $S_i=\{a+m\in S\ |\ e(a+m)=i(a+m)\}$ for $i=0,1,\lambda$. Moreover, we have the following relations
 $$ (S_0)^2\subseteq S_0,\quad (S_1)^2\subseteq S_1,\quad (S_{\lambda})^2\subseteq S_0\oplus S_1,$$
$$  S_{\lambda}S_0\subseteq S_{\lambda},\quad S_{\lambda}S_1\subseteq S_1,\quad S_0S_1=\{0\}.$$

On the other hand we have that $A_i=S_i\cap A$ and $M_i=S_i\cap M$, in fact $S_i\cap A=\{a+m\in S\ |\ e(a+m)=i(a+m)\}\cap A=\{a\in A\ |\ ea=ia\}=A_i$, similarly we have that  $M_i=S_i\cap M$.

\begin{itemize}
\item[]$A_0\cdot M_0 =(S_0\cap A)(S_0\cap M)\subseteq(S_0\cap M)=M_0,$
\item[]$A_1\cdot M_1 =(S_1\cap A)(S_1\cap M)\subseteq(S_1\cap M)=M_1,$
\item[]$A_{\lambda}\cdot M_{\lambda} =(S_{\lambda}\cap A)(S_{\lambda}\cap M)\subseteq(S_0\cap M)\oplus(S_1\cap M)=M_0\oplus M_1,$
\item[] $(S_0\cap A)(S_1\cap M)=(S_0\cap M)(S_1\cap A)=\{0\}$, that is $A_0\cdot M_1=A_1\cdot M_0 =\{0\}$.
\item[] $A_{\lambda}\cdot M_1=(S_{\lambda}\cap A)(S_{1}\cap M)\subseteq (S_{\lambda}\cap M)=M_{\lambda}$, similarly we have that  $A_1\cdot M_{\lambda}\subseteq M_{\lambda}$.
\item[] $A_{\lambda}\cdot M_0=(S_{\lambda}\cap A)(S_{1}\cap M)\subseteq (S_{\lambda}\cap M)=M_{\lambda}$. In a similar way we prove that  $A_0\cdot M_{\lambda}\subseteq M_{\lambda}$.
\end{itemize}

 If we add the conditions  $\beta\neq0$ and $\beta+3\gamma\neq 0$, we have that  $(S_{\lambda})^2=S_0 S_{\lambda}=\{0\}$, and the
relations $A_0\cdot M_{\lambda}=A_{\lambda}\cdot M_0=\{0\}$ and $ A_{\lambda}\cdot M_{\lambda}=\{0\}$ follow.
  \end{proof}

%------------------------------------------

 \section{Exceptional cases}
We now look at the three cases which arose as exception in Theorem \ref{teorema 1}.

\subsection{Case $\gamma=0$}
Let $A$ be a generalized almost-Jordan algebra and  $\gamma = 0. $ Then $\beta \neq 0,$ and  $A$ satisfies the identity  $(yx^2)x-((yx)x)x=0$, for every $x,y \in A.$ Let $\rho:A\rightarrow End(M)$ be a representation of $A$.
For these algebras the minimal polynomial of the operator  $R_e:A\rightarrow A$, is the same of the operator $\rho_e$ and it is given by $p(t)=t^2(t-1)$, (see identity (\ref{poli_car})). Then the Peirce decomposition of
 the algebra $A$ is $A=A_0\oplus A_1$, where $A_0=\{x\in A\ |\ (ex)e=0\}$ and  $A_1=\{x\in A\ |\ (ex)=x \}$. Similarly we have that $M=M_0\oplus M_1$, where $M_0=\{m\in M\ |\ \rho_e^2(m)=0\}$ and $M_1=\{m\in M\ |\ \rho_e(m)=m \}$.

Moreover, we have the relations (see \cite{CHP}).
$$A_0A_1\subseteq A_0,\ (A_0)^2\subseteq A_0.$$
\begin{lemma}\label{lemagamma}
Let $A$ be a generalized almost-Jordan algebra and  $\gamma = 0. $  Suppose that $A$ has an idempotent element $e\not=0$. Let $\rho:A \rightarrow End (M)$ be a representation of $A.$ Then
$$\ A_1\cdot M_0\subseteq M_0, A_0\cdot M_1\subseteq M_0,\ A_0\cdot M_0\subseteq M_0.$$
\end{lemma}

%\end{document}

\begin{proof}

As in the proof of Theorem \ref{teorema 1}, if we consider the algebra $S=A\oplus M$, then  $S$ satisfies identity (\ref{1}) for $\beta\not=0$ and $\gamma=0$.
Moreover $e$ is an idempotent element  in $S$, so in this case the Peirce decomposition
of $S$ relative to $e$ is $S=S_0\oplus S_1$, where $S_0=\{a+m\in S\ |\ e(e(a+m))=0\}$ and $S_1=\{a+m\in S\ | \ e(a+m)=a+m\}$.

On the other hand we have that $A_i=S_i\cap A$ and $M_i=S_i\cap M.$ In fact $S_1\cap A=\{a+m\in S\ |\ e(a+m)=(a+m)\}\cap A=\{a\in A\ |\ ea=a\}=A_1$
 and $S_0\cap A=\{a\in a\ |\ e(ea)=0 \}=A_0.$ Similarly we have that  $M_i=S_i\cap M$. Moreover, the subspaces $S_i$ satisfy the following relations
$$S_0S_1\subseteq S_0,\ (S_0)^2\subseteq S_0,$$
and we obtain that
\begin{itemize}
\item[] $A_0\cdot M_0=(S_0\cap A)(S_0\cap M)\subseteq (S_0\cap M)=M_0.$
\item[] $A_1\cdot M_0=(S_1\cap A)(S_0\cap M)\subseteq (S_0\cap M)=M_0.$
\end{itemize}

Similarly we obtain that $A_0\cdot M_1\subseteq M_0$ and Lemma 3.2 follows.

\end{proof}
%\end{document}

\subsection{Case $\beta+\gamma=0$}
Let $M$ be a vector space and $\rho:A\rightarrow End(M)$ a representation of $A$. We know by  identity (\ref{poli_car}) that the
minimal polynomial of $\rho_e$ and $R_e$ is $p(x)=x^2-x$. Then the Peirce decomposition of
 the algebra $A$ is $A=A_0\oplus A_1$, where $A_i=\{a\in A\ |\ ea=ia\}$ for $i=0,1$ . Similarly we have that $M=M_0\oplus M_1$, where $M_i=\{m\in M\ |\ \rho_e(m)=im\}$ for $i=0,1$.

We know that $A_0A_1=\{0\}$ and $(A_1)^2\subseteq A_1, $ (see \cite{CHP}).

%\end{document}
\begin{lemma}\label{caso2}
Let $A$ be a generalized almost-Jordan algebra and  $\beta +  \gamma = 0. $  Suppose that $A$ has an idempotent element $e\not=0$. Let $\rho:A \rightarrow End (M)$ be a representation of $A.$ Then
$$A_0\cdot M_1=A_1\cdot M_0=\{0\}, \ A_1\cdot M_1\subseteq M_1.$$
\end{lemma}
\begin{proof}

The algebra $S=A\oplus M$, is a generalized almost-Jordan algebra  so $S$ satisfies identity  (\ref{1}) for $\beta$ and $\gamma$ satisfying $\beta+\gamma=0$. Moreover the Peirce decomposition of $S$ relative to the idempotent is $S=S_0\oplus S_1$, where $S_i=\{a+m\in S\ |\ e(a+m)=i(a+m)\}$ for $i=0,1$.

          As the above case we have that  $A_i=S_i\cap A$ and $M_i=S_i\cap M$, and in this case we have the following relations
$$S_0S_1=\{0\},\ (S_1)^2\subseteq S_1.$$
 Therefore, we have that
\begin{itemize}
\item[] $A_1\cdot M_1=(S_1\cap A)(S_1\cap M)\subseteq (S_1\cap M)=M_1,$
\item[] $A_1\cdot M_0=(S_1\cap A)(S_0\cap M)=\{0\},$
\end{itemize}
Similarly we have that $A_0\cdot M_1=\{0\}$.
\end{proof}

\subsection{Case $\beta+2\gamma=0$}
Let $\rho:A\rightarrow End(M)$ be a representation of a generalized almost-Jordan algebra and  $\beta+2\gamma=0.$
For these algebras the minimal polynomial of the operator  $R_e:A\rightarrow A$, is the same of the operator $\rho_e$ and it is given by $p(t)=t(t-1)^2$, (see identity (\ref{poli_car})). Then the Peirce decomposition of
 the algebra $A$ is $A=A_0\oplus A_1$, where
$A_0=\{x\in A\ |\ (ex)=0\}$ and $A_1=\{x\in A\ |\ (ex)e-2(ex)+x=0 \}$. Similarly we have that $M=M_0\oplus M_1$, where $M_0=\{m\in M\ |\ \rho_e(m)=0\}$ and $M_1=\{m\in M\ |\ (\rho_e-id)^2(m)=0 \}$.

Moreover we have the following relations (see \cite{CHP}).
$$A_0A_1=\{0\},\ (A_0)^2\subseteq A_0.$$

\begin{lemma}\label{lemabeta+2gama}
Let $A$ be a generalized almost-Jordan algebra and  $ \beta +2 \gamma = 0. $  Suppose that $A$ has an idempotent element $e\not=0$. Let $\rho:A \rightarrow End (M)$ be a representation of $A.$ Then
$$A_0\cdot M_1= A_1\cdot M_0=\{0\},\ A_0\cdot M_0\subseteq M_0.$$
\end{lemma}
\begin{proof}

  The algebra $S=A\oplus M$, is a generalized almost-Jordan algebra so  $S$ satisfies identity (\ref{1}) for $\beta$ and $\gamma$ satisfying $\beta+2\gamma=0$. Moreover the Peirce decomposition of $S$ relative to the idempotent is $S=S_0\oplus S_1$, where $S_0=\{a+m\in S\ |\ e(a+m)=0\}$ and $S_1=\{a+m\in S\ | \ e(e(a+m))-2e(a+m)+ (a+m)=0\}$.

  As in the above Lemmas we have that $A_i=S_i\cap A$ and $M_i=S_i\cap M$. Moreover we have the following relations
$$S_0S_1=\{0\},\ (S_0)^2\subseteq S_0.$$
 Therefore we have
\begin{itemize}
\item[] $A_0\cdot M_0=(S_0\cap A)(S_0\cap M)\subseteq (S_0\cap M)=M_0,$
\item[] $A_1\cdot M_0=(S_1\cap A)(S_0\cap M)=\{0\}.$
\end{itemize}

Similarly we obtain that $A_0\cdot M_1=\{0\}$ and Lemma 3.6 follows.

\end{proof}

%------------------------------------------

\section{Irreducible Modules}
Let $A$ be an algebra over $K$ and  $\rho: A \rightarrow End(M)$ a representation of A.
\begin{definition} Let $N$ be a subspace of  $M$, we will say that $N$ is a submodule of $M$ if and only if $A\cdot N\subseteq N$.\end{definition}

\begin{definition} We will say that $M$ is an irreducible module or that $\rho$ is an irreducible representation, if $ M \neq 0$, and $M$
has no proper submodules  or equivalently there is no proper subspace of $M$ which are invariant under all the transformations $\rho(a), a \in A.$ \end{definition}

%\end{document}
\begin{proposition}\label{prop5}
 Let $A$ be a generalized almost-Jordan algebra and  $\beta, \gamma$ satisfying  $0\not \in \{\gamma,\beta+ \gamma,\beta+2\gamma \}$. Suppose that  $A$ has an idempotent element  $e\neq 0$. Let $\rho: A \rightarrow End(M)$ be an irreducible representation of $A. $
Then one of the following conditions hold.
$$(i) \; \; M=M_{\lambda} \;  \; \mbox{or} \; \;
 (ii) \; \; M=M_0 \;  \; \mbox{or} \; \;
  (iii) \; \; M=M_1$$
\end{proposition}
\begin{proof}
Using Theorem \ref{teorema 1}, we obtain that  $M_0$ and $M_{\lambda}$ are submodules of $M.$ Since $M$ is irreducible,
then $M=M_0$ or $M_0=\{0\}.$ On the other hand,  $M=M_{\lambda}$ or $M_{\lambda}=\{0\}$ and the Proposition follows. \end{proof}

\medskip
Linearizing  identity  (\ref{repre2})  we obtain
 \begin{multline}\label{linrepre2}
      (\beta +\gamma)(\rho_c\rho_{ab}+
      \rho_a\rho_{cb}+\rho_a\rho_c\rho_b+\rho_c\rho_a\rho_b+ \rho_{(ab)c}+\rho_{(cb)a})- \\ 2\beta(\rho_a\rho_b\rho_c +\rho_c\rho_b\rho_a+\rho_{(ac)b})- 2\gamma(\rho_b\rho_a\rho_c+ \rho_b\rho_c\rho_a+\rho_b\rho_{ac})=0
 \end{multline}

%\end{document}
\begin{theorem}\label{teoasc}
Let $A$ be a generalized  almost-Jordan algebra with $\beta$ and $\gamma$ satisfying $0\not \in \{\beta, \gamma, \beta+ \gamma, \beta+2\gamma, \beta+3\gamma\}$. Suppose that  $A$ has an idempotent element $e\neq 0$, and  $M$ be an irreducible module. If $M=M_1$, then $M$ is an associative module.
\end{theorem}

%\end{document}
\begin{proof}
 $M$ is associative if and only if
\begin{eqnarray}
% \nonumber to remove numbering (before each equation)
  (a,b,m) &=& 0 \qquad \forall a,b\in A,\ m\in M \label{asc1} \\
  (a,m,b) &=& 0 \qquad \forall a,b\in A,\ m\in M \label{asc2}
\end{eqnarray}
Since $M=M_1$ we have $\rho_e=id_M$.  We must prove relations  (\ref{asc1}) and (\ref{asc2}). Let $a,b\in A$ and $m\in M$,
then $a=a_0+a_1+a_{\lambda}$ and $b=b_0+ b_1+ b_{\lambda}$ with $a_i,b_i\in A_i \; \; i=0,1,\lambda.$ We have that
$$(a,b,m)=(a_0+a_1+a_{\lambda},b_0+ b_1+ b_{\lambda},m)=(a_1,b_1,m)$$
Similarly we have that $(a,m,b)=(a_1,m,b_1).$ Therefore for proving that $M$ is associative, we must verify relations (\ref{asc1}) and (\ref{asc2}) for all  $a,b\in A_1$ and $m\in M=M_1$.
$$(a,b,m) = ab \cdot m - a \cdot (b \cdot m) = \rho_{ab}(m) - \rho_a(\rho_b(m)) =  (\rho_{ab} - \rho_a\rho_b)(m) $$ and $$(a,m,b) = (a \cdot m) \cdot b - a \cdot (m \cdot b) = b  \cdot (a \cdot m) - a \cdot (b \cdot m) = (\rho_b\rho_a - \rho_a\rho_b)(m).$$ Therefore we need to prove that $\rho_b\rho_a = \rho_a\rho_b =\rho_{ab}.$

Replacing $a,b\in A_1$ and $c=e$ in identity (\ref{linrepre2})  we have
 $$ (\beta+\gamma)(3\rho_{ab}+3\rho_a\rho_b)-2\beta(\rho_a\rho_b+\rho_b\rho_a+\rho_{ab})-6\gamma(\rho_b\rho_a) = 0.$$
 Reordering the terms we have
  $$(\beta+3\gamma)\rho_{ab}+(\beta+3\gamma)\rho_a\rho_b-2(\beta+3\gamma)\rho_b\rho_a = 0.$$ Since $ \beta+3\gamma \neq 0$ we obtain
  \begin{equation}
   \rho_{ab}+\rho_a\rho_b-2\rho_b\rho_a= 0 \label{asc3}
\end{equation}

Interchanging  $a$ and $b$ in identity (\ref{asc3}) we obtain
\begin{equation}\label{asc4}
    \rho_{ab}+\rho_b\rho_a-2\rho_a\rho_b= 0
\end{equation}
Finally subtracting identity (\ref{asc3})  and  identity (\ref{asc4}) and using that char$(F) \neq 3,$ we obtain $\rho_a\rho_b=\rho_b\rho_a.$ Then we have that $\rho_{ab}=\rho_a\rho_b$. That is (\ref{asc1}) and (\ref{asc2}) are valid for all $a,b\in A_1$, and $M$ is an associative module.
\end{proof}

In the case $M=M_{\lambda}$ we have the following result
\begin{theorem}\label{teoasc2}
Let $A$ be a generalized almost-Jordan algebra with $\beta$ and $\gamma$ satisfaying $0\not \in \{\beta, \gamma, \beta+ \gamma, \beta+2\gamma, \beta+3\gamma,\beta-\gamma\}$. Suppose that $A$ has an idempotent element $e\neq 0$, and $M$ be an irreducible module. If $M=M_{\lambda}$, then the following relations hold
\begin{itemize}
\item[(i)]$(a,m,b)=0\qquad \forall\  a,b\in A,\ m\in M$
\item[(ii)]$(ab)m=\lambda^{-1}a(bm)\qquad \forall\  a,b\in A,\ m\in M$.
\end{itemize}
\end{theorem}
\begin{proof}
 Since $M=M_{\lambda}$ we have that  $\rho_e=\lambda id$. We must prove  (i) and (ii) for all $a,b\in A_1$ an $m\in M$. Replacing $a,b\in A_1$ and  $c=e$ in
identity (\ref{linrepre2}) we have that
%\begin{eqnarray}
% \nonumber to remove numbering (before each equation)
$$  (\beta+\gamma)(\lambda\rho_{ab}+\rho_a\rho_b+2\lambda\rho_a\rho_b+2\rho_{ab})-2\beta(\lambda\rho_a\rho_b+\lambda\rho_b\rho_a+\rho_{ab}) $$
  $$-2\gamma(2\lambda\rho_b\rho_a+\rho_b\rho_a) = 0 $$
 % \end{eqnarray}
  Reordering in term of $\rho_{ab}, \rho_a\rho_b $ and $\rho_b\rho_a$  we have
 $$ ((\beta+\gamma)(\lambda+2)-2\beta)\rho_{ab}+((\beta+\gamma)(2\lambda+1)-2\beta\lambda)\rho_a\rho_b$$
  $$-2(2\gamma\lambda+\gamma+\beta\lambda)\rho_b\rho_a=0 $$
Developing each coefficient and in the case of the coefficient of $\rho_{ab}$  we use the value of $\lambda,$  to get the identity
   \begin{equation}
      \gamma\rho_{ab}+(\beta+\gamma+2\gamma\lambda)\rho_a\rho_b-2\gamma\lambda \rho_b\rho_a=0 \label{nuevo}
\end{equation}
Interchanging  $a$ and $b$ in identity  (\ref{nuevo}) we have
\begin{equation}\label{nuevo1}
    \gamma\rho_{ab}+(\beta+\gamma+2\gamma\lambda)\rho_b\rho_a-2\gamma\lambda \rho_a\rho_b=0
\end{equation}
Subtracting  identity  (\ref{nuevo}) and identity (\ref{nuevo1}) we obtain
$$(\beta+\gamma+4\gamma\lambda)(\rho_a\rho_b-\rho_b\rho_a) = 0$$
Replacing the value of $ \lambda $ we obtain
\begin{equation}
\frac{(\beta+3\gamma)(\beta-\gamma)}{\beta+\gamma}(\rho_a\rho_b-\rho_b\rho_a)=
0
\end{equation}
Since $(\beta+3\gamma)\not=0$ and $(\beta-\gamma)\not=0$, we have $\rho_a\rho_b=\rho_b\rho_a.$ Therefore we have (i).  Using (\ref{nuevo}) we have
$$\gamma\rho_{ab}+(\beta+\gamma)\rho_a\rho_b=0$$ Since $ \beta + \gamma \neq 0,$ and using the value of $\lambda$ we obtain
\begin{equation}
-\lambda\rho_{ab} + \rho_a\rho_b=0
\end{equation}
Therefore $\rho_{ab}=\lambda^{-1}\rho_a\rho_b$,  we prove (ii), and the Theorem follows.\end{proof}

%------------------------------------------

\section{Exceptional cases}

We now look at five cases that arose as exception in Theorem \ref{teoasc} and in
Theorem \ref{teoasc2}

\subsection{Case $\beta=0$}

In this case, since $\beta=0$  using  Theorem  \ref{teorema 1} we have that  $M_0$ is submodule and we have the following result

\begin{lemma}
Let $A$ be a generalized almost-Jordan algebra with $\beta=0$. Suppose that $A$ has an idempotent element $e\neq 0$. Let $\rho: A \rightarrow End(M)$ be an irreducible representation of $A$. Then
$M=M_0$ or $M=M_1 \oplus M_{-1}$, where $M_i=\{m\in M\ |\ \rho_e(m)=im\}$ para $i=0,1,-1$.
\end{lemma}

\subsection{Case $\gamma=0$}

In this case,  using  Lemma   \ref{lemagamma} we have that  $M_0$ is submodule  and we have:

\begin{lemma}
Let $A$ be a generalized almost-Jordan algebra with $\gamma=0$. Suppose that $A$ has an idempotent element $e\neq 0$. Let $\rho: A \rightarrow End(M)$ be an irreducible representation of $A$. Then
$M=M_0$ or $M=M_1$, where $M_0=\{m\in M\ |\ \rho_e^2(m)=0\}$ y $M_1=\{m\in M\ |\ \rho_e(m)=m\}$.
\end{lemma}

\begin{proposition}
 Let $A$ be a generalized almost-Jordan algebra with $\gamma=0$. Suppose that $A$ has an idempotent element $e\neq 0$. Let $\rho: A \rightarrow End(M)$ be an irreducible  representation of  $A.$ If  $(A_1)^2\subseteq A_1$ and $M=M_1$, then $M$ is an associative module.
\end{proposition}

\begin{proof}
 Suppose $ \gamma = 0, (A_1)^2\subseteq A_1$ and $M=M_1$, that is $\rho_e = id_M.$  We need to prove that  $(a,b,m)=(a,m,b)=0$ for all  $a,b\in A_1, \; m\in M.$
 But $$(a,b,m) = ab \cdot m - a \cdot (b \cdot m) = \rho_{ab}(m) - \rho_a(\rho_b(m)) =  (\rho_{ab} - \rho_a\rho_b)(m) $$ and $$(a,m,b) = (a \cdot m) \cdot b - a \cdot (m \cdot b) = b  \cdot (a \cdot m) - a \cdot (b \cdot m) = (\rho_b\rho_a - \rho_a\rho_b)(m)$$

 Replacing  $a,b\in A_1$ and $c=e$ en relation (\ref{linrepre2}) we obtain $\rho_a\rho_b+\rho_{ab}-2\rho_b\rho_a=0$. Interchanging $a$ and $b$ in the above identity we obtain
  $\rho_b\rho_a+\rho_{ab}-2\rho_a\rho_b=0$. Subtracting  both identities we have that $\rho_a\rho_b =\rho_b\rho_a$. So $ \rho_{ab}=\rho_a\rho_b  $ and $M$ is an associative module.
\end{proof}

\subsection{Case $\beta+\gamma=0$}

In this case using  Lema  \ref{caso2} we have that $M_1$ is submodule of $M$ and we have the following result

\begin{lemma}\label{irr1}
 Let $A$ be a generalized almost-Jordan algebra with $\beta+\gamma=0$. Suppose that  $A$ has an idempotent element $e\neq 0.$ Let $\rho: A \rightarrow End(M)$ be an irreducible  representation of  $A.$ Then
$M=M_0$ \'o $M=M_1$.
\end{lemma}

\begin{proposition}
 Let $A$ be a generalized almost-Jordan algebra with $\beta+\gamma=0$. Suppose that  $A$ has an idempotent element $e\neq 0.$ Let $\rho: A \rightarrow End(M)$ be an irreducible  representation of  $A.$ If $(A_0)^2\subseteq A_0$ and $\rho_e\not=0$, then $M$  is an associative module.
\end{proposition}
\begin{proof}
As the above results we need to prove that  $(a,b,m)=(a,m,b)=0 \; \; \forall \; \; a,b\in A, m\in M$. Since $\rho_e\not=0$  Lema \ref{irr1} implies that $\rho_e=id_M$. Moreover we have that $(A_0)^2\subseteq A_0$,  so if $a=a_0+a_1$ and $b=b_0+b_1$ with $a_i,b_i\in A_i$, we have that
$(a,b,m)=(a_1,b_1,m)$ y $(a,m,b)=(a_1,m,b_1)$ and we only need to take  $a,b\in A_1$. With the same  argument using in the proof  of  Teorema \ref{teoasc} we prove that $M$ is an associative module.
\end{proof}

 The next example shows an irreducible module  of dimension $2,$ in the case $\beta+\gamma=0$
\begin{example}
Let us consider the algebra $A$ of base $\{e,a\}$  and multiplication table $e^2 = e, ea = ae = 0, a^2 = e^,$  given in  Example \ref{ej1}. Let  $M$ be a $2$-dimensional
 $\mathbb{R}$ - vector space $M$
and   $\{v,w\}$ a base of $M$. We define a linear map $\rho:A\rightarrow End(M)$ by $\rho_e=0$ and  $\rho_a(\lambda_1v+\lambda_2w)=(2\lambda_2-\lambda_1)v+(\lambda_2-\lambda_1)w$.
Then   $\rho$  satisfies (\ref{repre1}) and (\ref{repre2}), so  $\rho$ is a representation of A.
Suppose that $M$  is not  irreducible, that is,  there exists  a submodule $N=\mathbb{R}m$ for some $m\in M -\{0\}$. Let   $m=\lambda_1v+\lambda_2w\not=0$, since $N$ is a submodule of $M$,
we have that $\rho_x(m)=b_x m$ for some $b_x \in \mathbb{R}$, and for all  $x\in A$. Taking $x = a$ we have that  $\rho_a(m)=b_am$, and we obtain that
  % \nonumber to remove numbering (before each equation)
    $$(2\lambda_2-\lambda_1) = b_a\lambda_1, \; \;   (\lambda_2-\lambda_1)= b_a\lambda_2$$
  From the first identity we have $\lambda_2=\frac{(b_a+1)}{2}\lambda_1$, and replacing this value in the second identity we have
  \begin{eqnarray}
  % \nonumber to remove numbering (before each equation)
    \frac{(b_a+1)}{2}\lambda_1-\lambda_1 &=b_a\frac{(b_a+1)}{2}\lambda_1 \nonumber \\
    (b_a+1)\lambda_1-2\lambda_1 &=& b_a(b_a+1)\lambda_1 \nonumber\\
    ((b_a)^2+b_a-b_a-1+2)\lambda_1 &=& 0 \nonumber\\
    ((b_a)^2+1)\lambda_1 &=& 0 \nonumber
  \end{eqnarray}
Since the polynomial  $x^2+1=0$ is irreducible in $\mathbb{R}[x]$, we obtain that  $\lambda_1=0,$ and then $\lambda_2=0.$ A contradiction since  $m\not=0$. Therefore $M$ is a $2$-dimensional  irreducible
module.
\end{example}

\subsection{$\beta+2\gamma=0$}

 Lema \ref{lemabeta+2gama} implies that $M_0$ is un submodule of $M$, and we have

\begin{lemma}
 Let $A$ be a generalized almost-Jordan algebra with $\beta+2\gamma=0$. Suppose that  $A$ has an idempotent element $e\neq 0.$ Let $\rho: A \rightarrow End(M)$ be an irreducible  representation of  $A.$  Then
$M=M_0$ or $M=M_1$, where $M_0=\{m\in M\ |\ \rho_e(m)=0\}$ and $M_1=\{m\in M\ |\ (\rho_e-id)^2(m)=0\}$.
\end{lemma}

\subsection{$\beta+3\gamma=0$}
These algebras are the almost-Jordan algebras and it is known that for this kind of algebras every irreducible module is a Jordan module (see \cite{Sidorov}).

\bigskip

{\bf Open problems:} We do not know which is the situation with  an irreducible module $M$,
\begin{enumerate}
\item In the case $ M = M_0 $
\item In the case $\beta-\gamma=0,$ that is $A$ satisfies the identity,
$(yx^2)x + yx^3 -2((yx)x)x = 0.$
\item In the case $\beta=0,$ that is $A$ satisfies the identity, $ yx^3 -((yx)x)x = 0.$
\item In the case $\beta+2\gamma=0,$  that is $A$ satisfies the identity, $ yx^3 -2 (yx^2)x +((yx)x)x = 0.$
\end{enumerate}
In the last two cases we only know that $M_0$ is a submodule of $M.$

%------------------------------------------


\begin{thebibliography}{99}


\bibitem{AL}  Arenas, M., Labra, A. (2008). On nilpotency of generalized Almost-Jordan right-nilalgebras.  {\it Algebra Colloquium} 15 (1):69-82.

\bibitem{Arenas} Arenas, M. (2007). The Wedderburn principal theorem for generalized
almost-Jordan algebras.   {\it Comm. Algebra} 35  (2): 675-688.

\bibitem{BLR} Behn, A., Labra, A., Reyes C. (2013). Representations of Power-Associative Train Algebras. {\it Algebra Colloq. } Accepted in September.

\bibitem{BL} Benkart, G., Labra, A. (2006). Representations of algebras of rank 3.  {\it Communications in Algebra,} 34 (8): 2867-2877.

\bibitem{BIM} Bernard, J., Iltyacov, A., Martínez, C. (1994). {\it Bernstein representations}, Proceedings of the 3rd. Int. Conference on non associative algebras. (S. González, Ed.) Kluwer Academic Publisher, Dordrech. pp.39-45.

\bibitem{CHP}  Carini, L.,  Hentzel, I. R., Piacentini-Cattaneo, J. M. (1988). Degree four Identities not implies by commutativity.  {\it Comm. in Algebra} 16 (2):339-357.

\bibitem{ES 1}  Elduque, A., Shestakov, I. P. (1995). Irreducible Malcev modules.  {\it J. of Algebra} 173:622-637.

\bibitem{ES 2} Elduque, A., Shestakov, I. P (2001).  Irreducible non-Lie modules for Malcev superalgebras.  {\it J. of Algebra} 246:897-914

\bibitem{Eilenberg}  Eilenberg, S. (1948).  Extensions of general algebras.  {\it Ann. Soc. Polon. Math.} 21:125-134.

\bibitem{HL}    Hentzel, I. R.,  Labra, A. (2007). On left nilalgebras of left nilindex four satisfying an identity of degree four.  {\it Internat. J. Algebra Comput.}  17 (1):27-35.

\bibitem{HP} Hentzel, I. R.,  Peresi, L. A. (1988). Almost Jordan Rings.  {\it Proc. of A. M. S.}  104 (2):343-348.

\bibitem{Humphreys} Humphreys, J. E. (1972). {\it Introduction to Lie Algebras and Representation Theory}, Springer-Verlag New York Inc.

\bibitem{Jacobson 1} Jacobson, N. (1951). General representations theory of Jordan algebras.  {\it Trans. Amer. Math. Soc.} 70:509-530.

\bibitem{Jacobson 2} Jacobson, N. (1954). Structure of alternative and Jordan bimodules.  {\it Osaka Math. Journal} 6:1-71

\bibitem{Jacobson 3} Jacobson, N. (1968).  {\it Structure and representations of Jordan algebras.}  Amer. Math. Soc. Colloc. Publ. XXXIX: X + 453 pp.

\bibitem{LR} Labra, A. Reyes, C. (2005). Representations on train algebras of rank 3. {\it Linear Algebra and its Applic.} 400:91-97.

\bibitem{McCrimmon} McCrimmon, K. (1966). Bimodules for composition algebras.  {\it Proc. Amer. Math. Soc.} 17:480-486.

\bibitem{Osborn 1} Osborn, J. M. (1965). Commutative algebras satisfying an identity of degree four.  {\it Proc. Amer. Math. Soc.} 16:1114-1120.

\bibitem{Osborn 2} Osborn, J. M.  (1965). Identities of non-associative algebras.  {\it Canad. J. Math.} 17:78-92.

\bibitem{Osborn 3} Osborn, J. M. (1995) Modules for Novikov algebras.  {\it Contemporary Mathematics} 184:327-338.

\bibitem{Osborn 4} Osborn, J. M. (1971). Representations and radicals of Jordan algebras.  {\it Scripta Mathematica,} Vol. \textbf{29}, No. 3-4:  (1971), 297-329.

\bibitem{Petersson} Petersson, H. (1997). Zur Theorie der Lie-Tripel-Algebren.  {\it Math. Z.} 97:1-15.

\bibitem{Schafer 1} Schafer, R. (1952). Representations of alternatives algebras.  {\it Trans. Am. Math. Soc.} 72 (1):1-17.

\bibitem{Schafer 2} Schafer, R. (1966). Introduction to nonassociative algebras. Academic Press N. York, London.

\bibitem{Shestakov} Shestakov, I. P. (1979). Irreducible representations of Alternative algebras. Translated  from  {\it Mathematicheskie Zametki} 26(5): 673-686.

\bibitem{Sidorov} Sidorov, A. V. (1981).  Lie triple algebras. Translated from  {\it Algebra i Logika} 20 (1):101-108.



\end{thebibliography}
\end{document}